\documentclass[12pt]{article}

\usepackage{amsmath,amsthm,amsfonts,amssymb}
\usepackage{fullpage}
\usepackage{multirow}
\usepackage{array}
\usepackage{bbm}
\usepackage[noblocks]{authblk}
\usepackage{verbatim}
\usepackage{tikz}
\usepackage{float}
\usepackage{enumerate}
\usepackage{diagbox}
\usepackage{soul}
\usepackage{physics}
\usepackage{mathtools}
\usepackage[square,numbers]{natbib} 
\usepackage{relsize}%
\usepackage[framemethod=TiKz]{mdframed}

\newtheorem{theorem}{Theorem}[section]
\newtheorem{lemma}[theorem]{Lemma}

\newmdtheoremenv[nobreak=true]{notation}{Notation}

\usepackage{tcolorbox}
\usepackage{enumitem}

\theoremstyle{definition}
\newtheorem{definition}[theorem]{Definition}

\newlength{\Oldarrayrulewidth}

\newcommand{\cjk}[1]{\chi_j^k(#1)}
\newcommand{\mcijk}[1]{\lambda_j^k(#1)}

\newcommand{\cnek}[1]{\chi_e^k(#1)}
\newcommand{\ciek}[1]{\lambda_e^k(#1)}
\newcommand{\cnvk}[1]{\chi_v^k(#1)}
\newcommand{\civk}[1]{\lambda_v^k(#1)}

\newcommand{\cn}[1]{\rchi(#1)}   
\newcommand{\ci}[1]{\lambda(#1)}   

\newcommand{\Mod}[1]{\ (\mathrm{mod}\ #1)}

\DeclareRobustCommand{\rchi}{{\mathpalette\irchi\relax}}
\newcommand{\irchi}[2]{\raisebox{\depth}{$#1\chi$}} 

\DeclarePairedDelimiter\ceil{\lceil}{\rceil}
\DeclarePairedDelimiter\floor{\lfloor}{\rfloor}

\usepackage{geometry}
\usepackage{hyperref}

\title{Graph Coloring as a Measure of Network Vulnerability}

\author{
Grace Mulry\thanks{University of Texas, gracemulry@utexas.edu}
\hspace*{.5cm}
Mia DesStefano\thanks{Vassar College, mdestefano@vassar.edu}
\hspace*{.5cm}
Mason Nakamura\thanks{Marist College, mason.nakamura1@marist.edu}
\hspace*{.5cm}
Rodrigo Rios\thanks{Stony Brook University, rodrigoreyrios@gmail.com}
\hspace*{.5cm}
Nathan Shank \thanks{Moravian University, shankn@moravian.edu}
\hspace*{.5cm}

}

\date{}

\begin{document}

\maketitle

\begin{abstract}
We consider new parameters for conditional network vulnerability related to graph coloring. We define a network to be in operation if the chromatic number (or index) is greater than some specified value $k$.  The parameters of interest, the \textit{minimum $k$-chromatic number} and the \textit{minimum $k$-chromatic index} consider the least number of failures in the network which could render the network inoperable. In this paper, we consider edge failure, vertex failures, as well as mixed failures for paths, cycles, and complete graphs.  
\end{abstract}


\section{Introduction}
    Networks, modeled as graphs, can represent many things, including social networks, computational networks, electrical networks, neural networks, or transportation networks. The underlying structure of the network is vital to is operation. Network vulnerability parameters focus on investigating the stability of networks given a particular parameter of interest, which signifies if the network is operational or in a failure state. Network vulnerability can be thought of as the study of how much damage (e.g. number of removed vertices or edges) to a network must be done for it to enter a failure state. 

    Initially, in the 1970's and 80's, network vulnerability focused on connectivity. More recently, other parameters have been of interest. Harary \cite{Harary1983} generalized the idea for other models by considering a property $P$ and a network $G$ to be operational if there is a component of $G$ which contains property $P$.  Thus, a network is in a failure state if no component has property $P$. Some examples of parameters that have been studied include diameter \cite{shank_k-diameter_2018}, domination \cite{DominatingReliability}, and many variations on component order (see \cite{Gross2006}, \cite{gross_survey_2013}, and \cite{Gross2015} for example).  

    Through this paper we will assume that $G$ is a simple graph.  We will consider the chromatic number of $G$, denoted $\cn{G}$, as well as the chromatic index of $G$, denoted $\ci{G}$, to define the failure states. Recall that the chromatic number of a graph $G$ is the minimum number of colors needed to color the vertices of a graph so that no two adjacent vertices are the same color. The chromatic index of a graph is the minimum number of colors needed to color the edges of a graph such that no vertex is incident to two edges of the same color.  For other graph theory notation and terminology, we will follow \cite{West2000}.

    There are many applications of graph coloring, more specifically chromatic number and chromatic index. Chromatic number can assist in scheduling, organizing radio stations, traffic signaling, puzzle solver, and register allocation. Chromatic index has many applications in scheduling and communication networks (see \cite{Guichard} and \cite{Chaitin} for example).   

    However, investigating the vulnerability of networks based on these parameters also has practical applications, as it lets us know, quantitatively, the amount that a graph might be disturbed when variables change. For example, the study of network vulnerability can help inform how to construct the most reliable schedules possible for an organization in terms of the parameter; so, if one person is removed, the schedule still holds.



    In Chapter \ref{ChromaticsNumber}, we will focus on chromatic number, whereas Chapter \ref{ChromIndex} will focus on the chromatic index.  For each of these chapters, we will consider edge removal, followed by vertex removal, and then mixed removal where vertices are removed first and then edges. We will demonstrate the parameters with results for paths, cycles, and complete graphs. In chapter \ref{Conclusion}, we will briefly discuss some extensions as well as other interesting graph classes to consider. 

\section{Chromatic Number}\label{ChromaticsNumber}

    In this chapter, we will consider the chromatic number as the parameter of interest.  Section \ref{ChromNumberEdgeRemoval} will consider edge removal, Section \ref{ChromNumberVertexRemoval} will consider vertex removal, and Section \ref{ChromNumberMixedRemoval} will consider mixed removal.

    \subsection{Edge Removal}\label{ChromNumberEdgeRemoval}
        
        The goal of this parameter is to find the minimum number of edges that can be removed so that the remaining graph has a chromatic number less than $k$, for some positive integer $k$. To do this, we will first define a \textit{$k$-chromatic number edge removal set} as a set of edges, so that when removed, the resulting graph has a chromatic number less than or equal to $k$. Since we are trying to do this as efficiently as possible, the parameter of interest will be the size of the smallest $k$-chromatic number edge removal set. This is made more formal in the following definitions: 
        
        \begin{definition}
            Given a positive integer $k$ and a graph $G = (V,E)$, we say a set of edges $E' \subseteq E$ is a \textit{$k$-chromatic number edge removal set} if $\cn{G-E'} \leq k$.  
        \end{definition}
    
        \begin{definition}
            Let $\mathcal{E}$ be the set of all $k$-chromatic number edge removal sets.  Define the \textit{$k$-chromatic number edge removal parameter} as $$\cnek{G} = \min\{|E'|: E' \in \mathcal{E}\}.$$ 
        \end{definition}
    
        Any set $E'' \in \mathcal{E}$ where $|E''| = \cnek{G}$ will be called a \textit{minimum $k$-chromatic number edge removal set}. 
        
        For any graph $H$, if  $\cn{H}\leq k$, then $H$ is in a \textit{failure state.} If $\cn{H}>k$, then $H$ is not in a failure state and our network remains in an \textit{operable state}.  Thus, $\cnek{H} = 0$ if and only if $H$ is in a failure state.  
    
        If $\cn{G}=1$, then $|E(G)| = 0$.  Otherwise, any edge will require at least two , one for each of the vertices incident to that edge.  Thus we have the following simple lemma for the case when $k=1$. 
        
        \begin{lemma}\label{cne(k=1)}
          For any graph $G = (V,E)$, if $k = 1$, $\cnek{G} = |E|$.
        \end{lemma}
        
        We will now consider the path graphs, $P_n$, as well as cycle graphs, $C_n$.
        
        \begin{theorem}\label{cnekPath}
            If $P_n$ is a path graph on $n$ vertices then 
            $$\cnek{P_n}=\begin{cases}
                  0 &\text{ if } k>1,\\
                  n-1 &\text{ if } k=1.
            \end{cases}$$
        \end{theorem}
        \begin{proof}
            Notice that $\cn{P_n}=2$. Therefore, if $k>1$, $P_n$ is already in a failure state. The other case follows from Lemma \ref{cne(k=1)}.
        \end{proof}
    
        \begin{theorem}\label{cnekCycles}
            If $C_n$ is a cycle graph on $n$ vertices then 
            $$\cnek{C_n} = \begin{cases} 
                n &\text{ if } k=1,\\
                n\pmod{2} &\text{ if }  k=2, \\ 
                0 &\text{ if }  k>2. 
            \end{cases}$$
        \end{theorem}
        \begin{proof}
            Consider the cases where $k=2$ and $k>2$, as in the case where $k=1$ is Lemma \ref{cne(k=1)}.
            \begin{enumerate}
                \item Case $k=2$:
                    In $n$ is an even integers, it is known that $\cn{C_n}=2$ and thus $\cnek{C_n}=0$. If $n$ is an odd integer, then $\cn{C_n}=3$, and thus at least one edge removal is necessary. No additional edges need to be removed since $C_n - \{e\} = P_n$ which is already in a failure state according to Theorem \ref{cnekPath}.
                \item Case $k>2$:
                    Note that for all positive integers $n$, $\cn{C_n} \leq 3$ and so when $k>2$ (or $k\geq 3$) $C_n$ is already in a failure state.
            \end{enumerate}
        \end{proof}
    
        Now, we turn our attention to complete graphs, $K_n$.  Recall that $\cn{K_n} = n$.  If $k>n$, then $K_n$ is already in a failure state and $\cnek{K_n} = 0$. Therefore, we will only consider the case where $k\leq n$. 
        
        To proceed, we will need to use the well-known Turán graph.
        
        \begin{definition}
            For positive integers $n\geq k$, the Turán graph, denoted $T(n,k)=(V,E)$, as a complete multipartite graph with $k$ parts, $V_1, V_2, \ldots, V _k$, which partition $V$ such that $-1 \leq |V_i| - |V_j| \leq 1$ for all $1 \leq i \leq j \leq k$. 
        \end{definition}
        This is to say that $T(n,k)$ is a complete multipartite graph containing  $k$ parts of nearly equal size.  Note that $T(n,k)$ does not have a $K_{k+1}$ as a subgraph since there are only $k$ parts. We will denote the number of edges in the Turán graph by $|(T(n,k)|$.  Notice that if $p$ and $q$ are non-negative integers so that  $n=pk+q$ and $0 \leq q < k$ then $$|T(n,k)| = \binom{n}{2} - \left(q \binom{p+1}{2} + (k-q) \binom{p}{2}\right).$$
        
        \begin{theorem}\label{kChromaticNumberEdgeRemovalCompleteGraphs}
            Let $K_n$ be a complete graph on $n$ vertices.  For any positive integers $k \leq n$, find non-negative integers $p$ and $q$ so that $n=pk+q$ and $0 \leq q < k$.  Then,
            
            \[\cnek{K_n} = q \binom{p+1}{2} + (k-q) \binom{p}{2}.\]
        \end{theorem}
        \begin{proof}
        
            By Turán's Theorem (see \cite{West2000} pg 208), if $G$ is any graph of order $n$ with $|V(G)| > |T(n,k)|$, then $\cn{G} \geq k+1$ since $G$ must contains a $K_{k+1}$ as an induced subgraph. Since $\cn{T(n,k)} = k$, we can conclude that 
            \begin{align*}
                \cnek{K_n} &= |E(K_n)| - |T(n,k)|\\
                &=q \binom{p+1}{2} + (k-q) \binom{p}{2}.
            \end{align*}
        \end{proof}

    \subsection{Vertex Removal}\label{ChromNumberVertexRemoval}
    
        We will now turn our attention to the case of vertex removal. We proceed as before by considering all \textit{$k$-chromatic number vertex removal sets} and then finding the size of the minimum of these sets. 
        
        \begin{definition}
            Given a positive integer $k$ and a graph $G = (V,E)$, we say a set of vertices $V' \subseteq V$ is a \textit{$k$-chromatic number vertex removal set} if $\cn{G-V'} \leq k$.  
        \end{definition}
        
        \begin{definition}
            Let $\mathcal{V}$ be the set of all $k$-chromatic number vertex removal sets.  Define \textit{$k$-chromatic number vertex removal parameter} as $$\cnvk{G} = \min\{|V'|: V' \in \mathcal{V}\}.$$ 
        \end{definition}
            Any set $V'' \in \mathcal{V}$ where $|V''| = \cnvk{G}$ will be called a \textit{minimum $k$-chromatic number vertex removal set}. As before, any graph $H$ with $\cn{H}\leq k$ is in a \textit{failure state.}
        
        Below are the results for chromatic number with vertex removal for paths, cycles, and complete graphs. 
        
        \begin{theorem}\label{cnvkpaths}
            If $P_n$ is a path graph on $n$ vertices then
            $$\cnvk{P_n} =
                \begin{cases}
                    0 & \text{if $k\geq 2$},\\
                    \floor[\big]{\frac{n}{2}} & \text{if $k = 1$}.
                \end{cases}
            $$
        \end{theorem}
        \begin{proof}
            Consider the cases when $k\geq 2$ and $k=1$. 
            \begin{enumerate}
                \item Case $k\geq 2$:
                    Note that $\cn{P_n} \leq 2$.  Thus, $P_n$ is already in a failure state if $k\geq 2$.
                \item Case $k=1$: 
                    When $k=1$, removing every other vertex on the path will produce a graph which is one-colorable and thus in a failure state, so $\cnvk{P_n} \leq \floor[\big]{\frac{n}{2}}$.
                    Assume there is a $k$-chromatic number vertex removal set $V'$ with $|V'| = \floor[\big]{\frac{n}{2}} -1$. Each vertex removal inherently removes at most 2 edges.  So $P_n - V'$ would have $n-1 - 2\left(\floor[\big]{\frac{n}{2}} -1\right)\geq 1$ edges remaining and therefore has $\cn{P_n-V'}>1$. This implies that $\cnvk{P_n} \geq \floor[\big]{\frac{n}{2}}.$
            \end{enumerate}
        \end{proof}
        
        \begin{theorem}\label{cnvkcycles}
            If $C_n$ is a path graph on $n$ vertices then
            $$\cnvk{C_n} =
                \begin{cases}
                    0 & \text{if $k \geq 3$},\\
                    \left(n \Mod{2}\right) + (2-k)\lfloor \frac{n}{2}\rfloor &\text{if $k<3$}.\\
                \end{cases}
            $$
        \end{theorem}
        \begin{proof}
            Consider each case:
            \begin{enumerate}
                \item Case $k \geq 3$: 
                    Note that $\cn{C_n} \leq 3$.  Thus, $C_n$ is already in a failure state if $k \geq 3$.
                \item Case $k=2$ and $n$ is even: Note that $\cn{C_n} = 2$ and therefore is already in a failure state.  
                \item Case $k=1$ or $k=2$ and $n$ is odd:
                    The first vertex we remove will result in $P_{n-1}$ for which we can follow Theorem \ref{cnvkpaths}.  
            \end{enumerate}
            Note that the second two cases both produce $\cnvk{C_n} = \left(n \Mod{2}\right) + (2-k)\lfloor \frac{n}{2}\rfloor.$
        \end{proof}
        
        \begin{theorem}
            If $K_n$ is a complete graph on $n$ vertices then 
            $$ \cnvk{K_n} = 
                \begin{cases}
                    0 & \text{ if $k \geq n$},\\
                    n-k & \text{ if $k < n$}.
                \end{cases}
            $$
        \end{theorem}
        \begin{proof}
            If $k \geq n$, then $K_n$ is already in a failure state.  Note that for any $v\in V(K_n)$, $K_n-\{v\}=K_{n-1}$. From this, along with the fact that $\cn{K_i}=i$, it is clear that $\cnvk{K_n} = n-k$ for $k <n$. 
        \end{proof}
    
    \subsection{Mixed Removal}\label{ChromNumberMixedRemoval}
        In many applications, the vertices and edges are prone to failure. Therefore, Beineke and Harary \cite{HararyBeineke} introduced the \textit{connectivity function} of a graph that considered the mixed failure case. Since vertex removals have the potential to cause more than one edge removal, mixed connectivity assumes we remove vertices first and then remove edges. Hence, we need to have two parameters, one for the chromatic number which renders our network inoperable and one for the initial number of vertices we are allowed to remove. The goal in mixed removal, given these two parameters, is to find the minimum number of additional edges we must remove to produce a failure state. 
        
        As before, we will define edge sets that can be removed to satisfy our failure condition and then find the smallest size of one of these edge sets. 
        
        \begin{definition}
            Given positive integers $k$ and $j$ and a graph $G = (V,E)$, we say a set of edges $E' \subseteq E$ is a \textit{$k$-chromatic number $j$-mixed edge removal set} if there exists a set of vertices $V'\subseteq V$ with $|V'| = j$ so that $\cn{G-V'-E'} \leq k$.  
        \end{definition}
        
        \begin{definition}
            For any positive integer $j$, let $\mathcal{M}_j$ be the set of all $k$-chromatic number $j$-mixed edge removal sets. Define the \textit{$k$-chromatic index edge removal parameter} as$$\cjk{G} = \min\{|E'|: E' \in \mathcal{M}_j\}.$$ 
        \end{definition}
        
        Any set $E' \in \mathcal{M}_j$ where $|E'| = \cjk{G}$ will be called a \textit{minimum $k$-chromatic number $j$-mixed edge removal set}. 
        
        Note that for any graph, $G$, if $j \geq \cnvk{G}$, then $\cjk{G} = 0$ since  removing $j$ vertices can produce a failure state. We now demonstrate this parameter for paths, cycles, and complete graphs. 
        
        \begin{theorem}\label{cnmkpaths}
            If $P_n$ is a path graph on $n$ vertices then for any $0 \leq j \leq \floor[\big]{\frac{n}{2}}$ we have
            $$\cjk{P_n} = 
                \begin{cases}
                    n-1-2j & \text{ if $k=1$},\\
                    0 & \text{ if $k\geq 2$.}
                \end{cases}
            $$
        \end{theorem}
        \begin{proof}
            Observe that if $k\geq 2$, then $P_n$ is already in a failure state since $\cn{P_n}\leq 2$. If $k=1$, we must remove all edges to reach a failure state. Removing any $j$ vertices will remove at most $2j$ edges. Therefore, we need to remove at least $n-1-2j$ edges. This implies $\cjk{P_n} \geq n-1-2j$. 
            \begin{enumerate}
                \item Case 1: If $j < \floor[\big]{\frac{n}{2}}$ then we can remove every other vertex of degree 2 so that we remove exactly $2j$ edges. Thus, $\cjk{P_n} = n-1-2j$.
                \item Case 2:  If $j = \floor[\big]{\frac{n}{2}}$, then removing every other vertex will remove all $n-1$ edges. This implies $\cjk{P_n} = 0 = n-1-2j$. 
            \end{enumerate}
        \end{proof}
        
        \begin{theorem}\label{cnmkcycles}
            If $C_n$ is a cycle graph on $n\geq 3$ vertices then for any \\ $0 \leq j \leq \floor[\big]{\frac{n}{2}}$ we have
            $$\cjk{C_n} = 
                \begin{cases}
                    1 & \text{if $n$ is odd, $k = 2$, and $j=0$},\\
                    n-2j & \text{if $k=1$},\\
                    0 & \text{otherwise}.\\
                \end{cases}
            $$
        \end{theorem}
        \begin{proof}
            \begin{enumerate}
                \item Case 1: If $k \geq 3$ then $\cjk{C_n}=0$ since $\cn{C_n} \leq 3$. 
                \item Case 2: If $k=2$, the only case we need to consider are odd cycles since $\cn{C_n}=3$ when $n$ is odd.  If $n$ is odd, removing one vertex or edge will put $C_n$ in a failure state.  
                \item Case 3: If $k=1$ and $j\geq 1$ then, $$\cjk{C_n}=\chi_{j-1}^k(P_{n-1}), $$ and the result holds from Theorem \ref{cnmkpaths}.  
                \item Case 4: If $k=1$ and $j=0$, then $\cjk{C_n} = \cnek{C_n}=|E| = n$.
            \end{enumerate}
        \end{proof}
        
        \begin{theorem}\label{cnmkcompletes}
            If $K_n$ is a complete graph on $n$ vertices then for any \\ $0 \leq j \leq n-k$ we have
            $$\cjk{K_n} = 
                \begin{cases}
                    0 & \text{ if $k \geq n$},\\
                    |E(K_{n-j})|-|T(n-j,k)| & \text{ if $k<n$}.
                \end{cases}
            $$
        \end{theorem}
        \begin{proof} 
            If $k\geq n$, then $\cjk{K_n} = 0$ since $K_n$ is already in a failure state. 
            If $k<n$, the removal of any $j$ vertices will result in $K_{n-j}$.  Hence $\cjk{K_n} = \cnek{K_{n-j}}$ and the result follows by Theorem \ref{kChromaticNumberEdgeRemovalCompleteGraphs}
        \end{proof}

\section{Chromatic Index}\label{ChromIndex}

    We will now turn our attention to chromatic index, which involves coloring the edges rather than the vertices.  As before, Section \ref{ChromIndexEdgeRemoval} will consider edge removal, Section \ref{ChromIndexVertexRemoval} will consider vertex removal, and Section \ref{ChromIndexMixedRemoval} will consider mixed removal.
    \subsection{Edge Removal}\label{ChromIndexEdgeRemoval}
    
    As before, we will now define $k$-chromatic index edge removal sets and the minimum size of these sets will be our parameter of interest.  This parameter seems to fit well with edge removal since we are coloring the edges.  
    
        \begin{definition}
            Given a positive integer $k$ and a graph $G = (V,E)$, we say a set of edges $E' \subseteq E$ is a \textit{$k$-chromatic index edge removal set} if $\ci{G-E'} \leq k$.  
        \end{definition}
    
        \begin{definition}
            Let $\mathcal{E}$ be the set of all $k$-chromatic index edge removal sets.  Define the \textit{$k$-chromatic edge removal parameter} as $$\ciek{G} = \min\{|E'|: E' \in \mathcal{E}\}.$$ 
        \end{definition}
        
        Any set $E'' \in \mathcal{E}$ where $|E''| = \ciek{G}$ will be called a \textit{minimum $k$-chromatic index edge removal set}. 

        Given a proper coloring of the graph, we can easily see that an upper bound for our parameter would be the number of edges that are colored with a value larger than $k$.  This would produce a failure state. However, since association of colors with numerical values is somewhat arbitrary, this does not always give the fewest number of edges that might be removed. This will be explored briefly in Chapter \ref{Conclusion}.
        
        We now demonstrate how the $k$-chromatic index edge removal parameter works on paths, cycles, and complete graphs. 
        
        \begin{theorem}\label{ciekpaths}
            If $P_n$ is a path graph on $n$ vertices then
                \[
                    \ciek{P_n}=\begin{cases}
                    0  & \text{if} \ k > 1,\\
                    \floor[\big]{\frac{n-1}{2}} & \text{if} \ k = 1.\\
                    \end{cases}
                \]
        \end{theorem}
        \begin{proof}
            Since $\ci{P_n} \leq 2$ for $n\in \mathbb{N}$, if $k >1$, then $P_n$ is already in a failure state.  
            However, if $k = 1$, then if we remove every other edge along the path starting with an edge between two vertices of degree two, we can see that $\ciek{P_n} \leq \floor[\big]{\frac{n-1}{2}}$. Notice that if $E'$ was a minimum $1$-chromatic index edge removal set, then $P_n-E'$ can not have a vertex of degree 2. Notice that $P_n$ has $n-2$ vertices of degree 2. Therefore, we need to remove at least $ \ceil[\big]{\frac{n-2}{2}} $ edges. But, $\ceil[\big]{\frac{n-2}{2}} \geq \floor[\big]{\frac{n-1}{2}}$ which completes the proof.
        \end{proof}
        
        \begin{theorem}\label{ciekcycles}
           If $C_n$ is a cycle graph on $n$ vertices then
                 \[
                    \ciek{C_n}=\begin{cases}
                    0  & \text{if} \ k > 2,\\
                    n \Mod{2} & \text{if} \ k=2,\\ 
                    1+ \floor[\big]{\frac{n-1}{2}} & \text{if} \ k = 1.\\
                    \end{cases}
                \]
        \end{theorem}
        \begin{proof}
            Since $\ci{C_n} \leq 2$ for $n$ even and $\ci{C_n} \leq 3$ for $n$ odd, we examine the cases for differing $n$'s.  If $k > 2$, then $C_n$ is already in a failure state. 
            \begin{enumerate}
                 \item Case $k>2$:
                    Note that for all positive integers $n$, $\ci{C_n} \leq 3$ and so when $k>2$ (or $k\geq 3$) $C_n$ is already in a failure state.
                \item Case $k=2$:
                    If $n$ is an even integer, it is known that $\ci{C_n}=2$ and thus $\ciek{C_n}=0$. If $n$ is an odd integer, then $\ci{C_n}=3$ and thus at least 1 edge removal is necessary. No additional edges need to be removed since $C_n - \{e\} = P_n$ which is already in a failure state by \ref{ciekpaths}.
                \item Case $k=1$: 
                    Because $P_n = C_n - \{e\}$ for any edge $e$ we have $\ciek{C_n} = 1+ \ciek{P_n}$. 
            \end{enumerate}
        \end{proof}
        
        \begin{theorem}\label{ciekcompletes}
            For any positive integer $k \leq n$ we have 
                \[
                    \ciek{K_n}=\begin{cases}
                    0  & \text{if} \ k \geq n,\\
                    \binom{n}{2} - \left\lfloor{\frac{n}{2}}\right\rfloor k & \text{if} \ k < n.\\
                    \end{cases}
                \]
        \end{theorem}
        \begin{proof}
            Note that if $E' \subseteq E(K_n)$ is a minimum $k$-chromatic index edge removal set so that $|E'| \leq \binom{n}{2} - \floor{\frac{n}{2}}k$, then $|E(K_n - E')| > \floor{\frac{n}{2}}k$.  This would imply that a $k$-edge coloring of $K_n-E'$ would have at least one color used more than $\floor{\frac{n}{2}}$ times and therefore is not a proper edge coloring. This implies that $\ciek{K_n} \geq \binom{n}{2} - \floor{\frac{n}{2}}k.$
            Let $n' = \left\lfloor{\frac{n+1}{2}}\right\rfloor 2-1$.  So $n' = n$ if $n$ is odd and $n' = n-1$ if $n$ is even. In this way, we know that $\lambda(K_n) = n'$ and we will label these colors $\{a_1, a_2, \ldots, a_{n'}\}$.  Notice that any color is used exactly $\left\lfloor{\frac{n}{2}}\right\rfloor$ times.  
            Consider the spanning subgraph $H$ of $K_{n}$ which includes all the edges whose colors are contained in the set $\{a_1, a_2, \ldots, a_k\}$.  Thus, $H$ is $k$-edge colorable and, therefore, $E(K_n) - E(H)$ is a $k$-chromatic index edge removal set.  Since $|E(H)| = \left\lfloor{\frac{n}{2}}\right\rfloor k$ and $H$ is a subgraph of $K_n$ we have 
            \begin{align*}
                \ciek{K_n} &\leq |E(K_n) - E(H)|\\
                        & = \binom{n}{2} - \left\lfloor{\frac{n}{2}}\right\rfloor k.\\
            \end{align*}
        \end{proof}
 
    \subsection{Vertex Removal}\label{ChromIndexVertexRemoval}
    
    Now we consider the chromatic index with vertex removal. 
    
        \begin{definition}
            Given a positive integer $k$ and a graph $G = (V,E)$, we say a set of vertices $V' \subseteq V$ is a \textit{$k$-chromatic index vertex removal set} if $\ci{G-V'} \leq k$.  
        \end{definition}
    
        \begin{definition}
            Let $\mathcal{V}$ be the set of all $k$-chromatic index vertex removal sets. Define the \textit{$k$-chromatic index vertex removal parameter} as $$\civk{G} = \min\{|V'|: V' \in \mathcal{V}\}.$$ 
        \end{definition}
        
        Any set $V'' \in \mathcal{V}$ where $|V''| = \civk{G}$ will be called a \textit{minimum $k$-chromatic index vertex removal set}. 

        We now demonstrate how this works on paths, cycles, and complete graphs. 

        \begin{theorem}\label{civk paths}
            If $P_n$ is a path graph on $n$ vertices then
                \[
                    \civk{P_n}=\begin{cases}
                    0  & \text{if} \ k > 1,\\
                    \floor{\frac{n}{3}}& \text{if} \ k = 1.\\
                    \end{cases}
                \]
        \end{theorem}
        \begin{proof}
            It is easy to see that $P_3$, $P_4$, and $P_5$ all have a minimum $k$-chromatic index edge vertex removal set of size 1.  
            Consider a path $P_n$, $n\geq 6$, where $V = \{v_1, \ldots, v_n\}$.  Note that $P_n-\{v_{n-2}\} = P_{n-3} + P_2.$ So it is clear that $\civk{P_n} \leq \civk{P_{n-3}} + 1 = \floor{\frac{n-3}{3}}+1 = \floor{\frac{n}{3}}.$
            Notice that any $k$-chromatic index vertex removal set, $V'$, for $P_n$, must contain at least one vertex from the set $\{v_{n-2}, v_{n-1}, v_n\}$; call this vertex $v$.  Then $P_n-\{v\}$ contains a path of length at least $n-3$ and therefore we must remove at least $\floor{\frac{n-3}{3}}$ additional vertices. So $\civk{P_n} \geq 1+ \floor{\frac{n-3}{3}} = \floor{\frac{n}{3}}$
        \end{proof}

       \begin{theorem} 
            Given any positive integer $n \geq 3$ we have 
            \begin{equation*}
            \civk{C_n} = 
            \begin{cases}
                0 & \text{ if } k \geq 3,\\
                n \Mod{2} & \text{ if $k=2$},\\
                1+\civk{P_{n-1}} &\text{ if } k=1.
            \end{cases}
            \end{equation*}
        \end{theorem}
        \begin{proof}
            Following the proof for Theorem \ref{ciekpaths}, we will only consider the following cases: 
            \begin{enumerate}
                \item Case $k=2$: 
                    If $k = 2$ and $n$ is even, an even cycle will have a chromatic index of 2, so no vertices must be removed. If $k =2$ and $n$ is odd, 1 vertex must be removed which will result in $P_{n-1}$ and paths have a chromatic index of 2.
                \item Case $k=1$: 
                    Because $P_{n-1} = C_n - \{v\}$ for any vertex $v$ we have 
                    $\civk{C_n}  = 1+\civk{P_{n-1}} =  1 + \floor{\frac{n-1}{3}}. $\\
            \end{enumerate}
        \end{proof}
    
        \begin{theorem}\label{cievk K_n} 
            For any positive integer $k$,
            \begin{equation*}
                \lambda_v^k(K_n) = \begin{cases}
                0 & \text{if $k \geq n$}, \\
                n - k - \big{(}k\Mod{2}\big{)} & \text{if $k < n$}.
                \end{cases}
            \end{equation*}
        \end{theorem}
        \begin{proof}
            Note that $\ci{K_n} = n-1$ is $n$ is even and $\ci{K_n} = n$ if $n$ is odd.  Consider first when $k \geq n$. Then, we know $\ci{K_n} = n$ or $n -1$, so either way the graph is already in a failure state and $\lambda_v^k(K_n) = 0$. 
            Now consider when $k < n$ and $k$ is odd. We also note that $n > 2$ since $\ci{K_2} = 1$, which is already a failure state. When $k$ is odd, we know that a $K_k$ will have a chromatic index of $k$. However, we also know that $K_{k+1}$ will have a chromatic number of $k$ since $k+1$ is even. Because we want to find the minimum amount of vertices necessary to remove to put the graph in a failure state, we will want to remove all except $k + 1$. So, we have $\civk{K_{n}} = n - (k + 1)$ when $k$ is odd. 
            We now consider when $k<n$ and $k$ is even. Notice that it is impossible to obtain a graph with a chromatic index of exactly $k$, so instead we want to obtain a graph with a chromatic index of $k-1$, and do so as efficiently as possible. For even $k$ values, $\ci{K_{k+1}} = \ci{K_k} + 2$. So, since $\ci{K_k} = k - 1$, we can see that the quickest way to attain a chromatic index of less than $k$ is by removing all but $k$ vertices. That is, $\civk{K_n} = n - k$ when $k$ is even. 
        
        
        \end{proof}
    
    \subsection{Mixed Removal}\label{ChromIndexMixedRemoval}
        
        When we consider mixed removal, we will always be given two positive integer parameters, $j$ and $k$ and proceed by removing $j$ vertices first and then removing edges until we reach a failure state, indicated by a chromatic index less than or equal to $k$.  The smallest edge set, taken over all sets of $j$ vertices is the minimum $k$-chromatic index $j$-mixed edge removal.  
        
         \begin{definition}
            Given positive integers $k$ and $j$ and a graph $G = (V,E)$, we say a set of edges $E' \subset E$ is a \textit{$k$-chromatic index $j$-mixed edge removal set} if there exists a set of vertices $V'\subset V$ with $|V'| = j$ so that $\ci{G-V'-E'} \leq k$. 
        \end{definition}
        
        \begin{definition}
            For any positive integer $j$, let $\mathcal{N}_j$ be the set of all $k$-chromatic index $j$-mixed edge removal sets. Define the \textit{$k$-chromatic index $j$-mixed edge removal parameter} as 
            $$\mcijk{G} = \min\{|E'|: E' \in \mathcal{N}_j\}.$$ 
        \end{definition}
        
        Any set $E' \in \mathcal{N}_j$ where $|E'| = \cjk{G}$ will be called a \textit{minimum $k$-chromatic index $j$-mixed edge removal set}. 
        
        As before, we demonstrate this parameter on paths, cycles, and complete graphs.

        \begin{theorem}\label{cimkpaths}
            If $P_n$ is a path graph on $n$ vertices then for any $0 \leq j \leq \floor[\big]{\frac{n}{2}}$ we have
                \[
                    \mcijk{P_n} = \begin{cases}
                        \floor[\big]{\frac{n-3j-1}{2}} & \text{ if $k=1$},\\
                        0 & \text{ if $k\geq 2$ or $n\leq 3j+2$.}
                    \end{cases}
                \]
        \end{theorem}
        \begin{proof}
            Observe that if $k\geq 2$, then $P_n$ is already in a failure state since $\ci{P_n}\leq 2$.
            
            For $k=1$, note that if $n\leq3j+2$, then $\mcijk{P_n} = 0$ since $\civk{P_n} = \floor{\frac{n}{3}} \leq \floor{\frac{3j+2}{3}} = j$.  Assume $n \geq 3j+3$. The failure state must consist of isolated copies of $P_2$. Let $P_n$ be a path graph defined in the standard way so that $V(P_n) = \{v_1, v_2, \ldots v_n\}$ with $v_i$ adjacent to $v_{i+1}$. Consider $V' = \{v_3, v_6, \ldots v_{3j}\}$. Then $P_n - V'$ consists of isolated copies of $P_2$ and a single copy of $P_{n-3j}.$ Since $\ciek{P_{n-3j}} = \floor[\big]{\frac{n-3j-1}{2}}$, we see $\mcijk{P_n} \leq \floor[\big]{\frac{n-3j-1}{2}}.$
            For a path graph, the total degree is $2(n-1)$.  If we remove $j$ vertices and $m$ edges from $P_n$, then the total degree must be less than or equal to $2(n-1) - 4j - 2m$ which implies the average degree must be less than or equal to $\frac{2(n-1) - 4j - 2m}{n-j}.$  If $H \subset P_n$ is in a failure state, then the average degree in $H$ must be at most 1. Thus, $\frac{2(n-1) - 4j - 2m}{n-j} \leq 1$, which is equivalent to 
            $\frac{n-3j-1}{2} \leq m + \frac{1}{2}$.  Taking floors of both sides and noting that $m$ is already an integer we have $\floor[\big]{\frac{n-3j-1}{2}} \leq m$. So $\mcijk{P_n} \geq \floor[\big]{\frac{n-3j-1}{2}}$ which completes the proof. 
        \end{proof}
        
        \begin{theorem}\label{cimkcycles}
            If $C_n$ is a cycle graph on $n$ vertices then for any $0 \leq j \leq \floor[\big]{\frac{n}{2}}$ we have
            \[
                \mcijk{C_n} = \begin{cases}
                    \ciek{C_n} &\text{ if $j = 0$},\\
                    0 &\text{ if $k \geq 3$},\\
                    0 &\text{ if $k=2$ and $j\geq 1$},\\
                    \lambda_{j-1}^k(P_{n-1})  &\text{ if $k=1$ and $j \geq 1$}.\\
                \end{cases}
            \]
        \end{theorem}
        \begin{proof}
            Notice that if $j=0$ then we just need to consider $\ciek{C_n}$ which is Theorem \ref{ciekcycles}.
            If $j\geq 1$ and $k=2$, then we are either in a failure state ($n$ even) or by removing a vertex we will be in a failure state ($n$ odd). 
            If $k=1$ and $j \geq 1$, the first vertex removal will result in $P_{n-1}$ so we can consider mixed removal on $P_{n-1}$ which is Theorem \ref{cimkpaths}.
        \end{proof}

        \begin{theorem}\label{cimkcompletes}
            If $K_n$ is a complete graph on $n$ vertices then for any $0 \leq j \leq \floor[\big]{\frac{n}{2}}$ we have
            \[
                \mcijk{K_n} = \begin{cases}
                0 & \text{ if $k\geq n-j$},\\
                \binom{n-j}{2} - \floor{\frac{n-j}{2}}k & \text{ if $k < n-j$}.
                \end{cases}
            \]
        \end{theorem}
        \begin{proof}
            If $k\geq n-j$ consider two additional cases of $n-j \leq k < n$ and $n \leq k$. If it is the case that $n \leq k$ then it it follows,
                $$\ci{K_n}\leq n \leq k.$$
            If it is the case that $n-j \leq k < n$, the removal of $j$ vertices will result in $K_{n-j}$ at which,
                $$\ci{K_{n-j}}\leq n-j \leq k. $$
        When $k<n-j$ invoke the appropriate case of Theorem \ref{ciekcompletes}
        \end{proof}

\section{Conclusion}\label{Conclusion}

    Note that if we are trying to use at most $k$ colors, we could put our network into a failure state by removing any vertex (edge) which has a label greater than $k$.  However, this is not always the optimal solution and the numerical values of for the colors are arbitrary. See, for example, consider the graph $G$ in Figure \ref{VertexRemoval} whose vertices are colored using the numbers $\{1, 2, 3\}$.  Notice that chromatic number of the graph is three and each of the colors must be used at least three times to properly color the graph.  However, $\rchi_v^2{G} =2$ by removing the two vertices labeled $a$ and $b$ which leaves four disjoint $P_2$ graphs.

    \begin{figure}[h!]
        \centering
        \includegraphics[width = .3\linewidth]{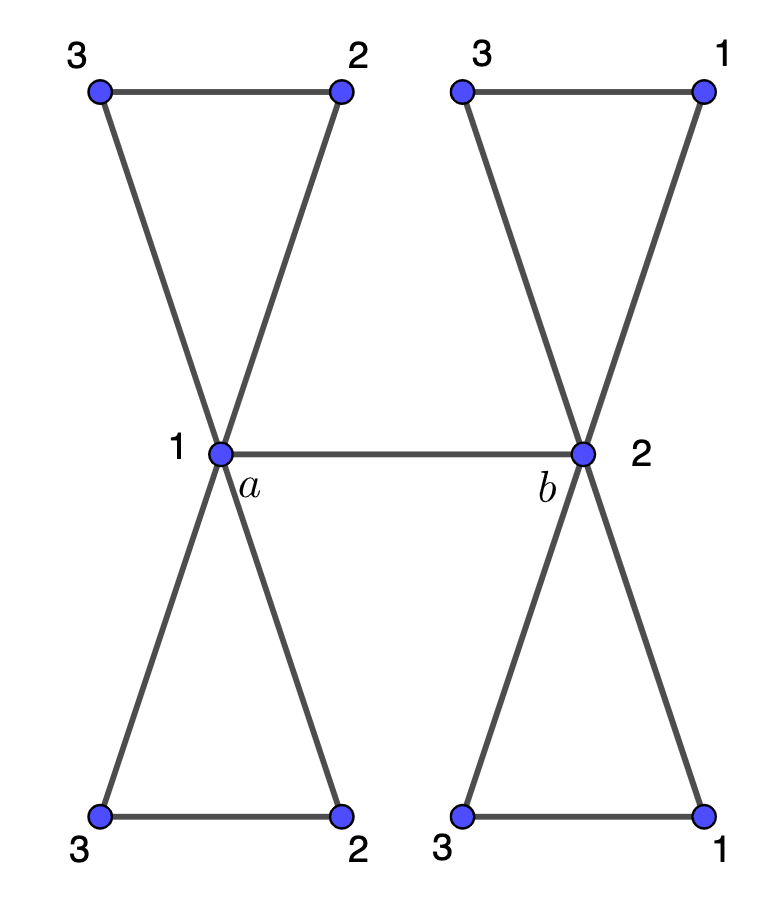}
        \caption{Graph $G$ with $\cn{G}=3$. }
        \label{VertexRemoval}
    \end{figure}

    So far, our attention has been limited to demonstrating these network vulnerability parameters for paths, cycles, and complete graphs. Unsurprisingly finding a solution for paths naturally leads to a solution for cycles.  Complete graphs require finding the particular graph construction given a particular order, $n$, and size, $m$, without going over a specific coloring.  In this way, it is more so a question about extremal graph theory.

    These graph classes are easy enough to analyze since their structures are uniquely determined by some parameter $n$ (mainly number of vertices). There are other such classes such as wheels, stars, and friendship graphs, whose construction is explicitly determined by some parameters which could also be studied. However, consideration could also be given to study the chromatic network vulnerability parameters outlined here on other graph classes whose construction is not as explicit. For example, trees, $r$-regular graphs, and planar graphs should be considered.   There are also other coloring parameters that may be more intuitively associated with reliability that can not be studied.  For example, complete chromatic number with mixed removal is a natural extension.  

    Another interesting extension would be to consider the vulnerability parameter for other classes of graphs, specifically $G(n,m)$ which is the set of graph with order $n$ and size $m$.  Finding the minimum and maximum value of our parameters over the set $G(n,m)$ will tell us the least and most reliable networks we can construct for a given size and order.  Using data generated from some python script \href{https://github.com/rodrigoReyrios/NetworkReliability2}{here}, we have been able to find certain minimums and explore certain maximums.  In some way the question about maximum is more difficult since it requires finding the least optimum removal set over a set of already optimized removals.


\bibliographystyle{plain}
\bibliography{GraphColorV1}

\begin{thebibliography}{10}

\bibitem{HararyBeineke}
L.~Beineke and F.~Harary.
\newblock The connectivity function of a graph.
\newblock {\em Mathematika}, 14:197--202, 1967.

\bibitem{Gross2006}
F.~Boesch, D.~Gross, W.~Kazmierczak, C.~Suffel, and A.~Suhartomo.
\newblock Component order edge connectivity-an introduction.
\newblock {\em Proceedings of the Thirty-Seventh Southeastern International
  Conference on Combinatorics, Graph Theory and Computing - Conger. Numen.},
  178:7--14, 2006.

\bibitem{shank_k-diameter_2018}
A.~Buzzard and N.~Shank.
\newblock The k-diameter component edge connectivity parameter.
\newblock {\em Involve, A Journal of Mathematics}, 11(5):845--856, April 2018.
\newblock Publisher: Mathematical Sciences Publishers.

\bibitem{Chaitin}
G.~J. Chaitin.
\newblock Register allocation and spilling via graph coloring.
\newblock {\em SIGPLAN Not.}, 17(6):98–101, June 1982.

\bibitem{DominatingReliability}
K.~Dohmen and P.~Tittmann.
\newblock Domination reliability.
\newblock {\em Electron. J. Combin.}, 19(1), 2012.

\bibitem{gross_survey_2013}
D.~Gross, M.~Heinig, L.~Iswara, W.~Kazmierczak, K.~Luttrell, J.~Saccoman, and
  C.~Suffel.
\newblock A survey of component order connectivity models of graph theoretic
  networks.
\newblock {\em WSEAS Transactions on Mathematics}, 12:895--910, September 2013.

\bibitem{Gross2015}
D.~Gross, M.~Heinig, J.~Saccoman, and C.~Suffel.
\newblock On neighbor component order edge connectivity.
\newblock {\em Congressus Numerantium}, 223:17 -- 32, 01 2015.

\bibitem{Guichard}
D.~Guichard.
\newblock An introduction to combinatorics and graph theory.
\newblock Open Educational Resource (OER).

\bibitem{Harary1983}
F.~Harary.
\newblock Conditional connectivity.
\newblock {\em Networks}, 13:347--357, 1983.

\bibitem{West2000}
D.~B. West.
\newblock {\em Introduction to Graph Theory}.
\newblock Prentice Hall, 2 edition, September 2000.

\end{thebibliography}


\end{document}